\documentclass[12pt]{amsart}

\usepackage{amsfonts,amsmath,amssymb}

\usepackage{enumerate}

\numberwithin{equation}{section}

\newtheorem{theorem}{Theorem}[section]

\newtheorem{lemma}[theorem]{Lemma}

\newtheorem{proposition}[theorem]{Proposition}

\newtheorem{example}[theorem]{Example}
\newtheorem{examples}[theorem]{Examples}

\newtheorem{definition}{Definition}[section]

\newtheorem{corollary}[theorem]{Corollary}

\newtheorem{remark}[theorem]{Remark}

\newcommand{\cl}[1]{\mathcal{#1}} 

\newcommand{\bb}[1]{\mathbb{#1}}

\newcommand{\Alg}[1]{\mathrm{Alg}(#1)} 

\begin{document}

 \title{ Morita embeddings for dual operator algebras and dual operator spaces}

\author[G. K. Eleftherakis ]{G. K. Eleftherakis }

\address{G. K. Eleftherakis\\ University of Patras\\Faculty of Sciences\\ Department of Mathematics\\265 00 Patras Greece }
\email{gelefth@math.upatras.gr}



\maketitle

 \begin{abstract} We define a relation $\subset _\Delta $ for dual operator algebras. We say that $B\subset _\Delta A$ if
 there exists a projection $p\in A$ such that $B$ and $pAp$ are Morita equivalent in our sense. We show that $\subset _\Delta $ is
 transitive, and we investigate the following question: If $A\subset _\Delta B$ and $B\subset _\Delta A$, then is it true that $A$
 and $B$ are stably isomorphic? We propose an analogous relation $\subset _\Delta $ for dual operator spaces, and we present some
 properties of $\subset _\Delta $ in this case. \end{abstract}

\section{Introduction} 

An operator space $X$ is said to be a dual 
operator space if $X$ is completely isometrically isomorphic to the operator space dual $Y^*$ of an operator space $Y.$ If, in addition, 
$X$ is an operator algebra, then we call it a dual operator algebra. For example, Von Neumann algebras and 
nest algebras are dual operator algebras.
 Blecher, Muhly and Paulsen introduced the notion of the Morita equivalence of non-self-adjoint operator algebras \cite{bmp}. 
Subsequently, Blecher and Kashyap developed a parallel theory for dual operator algebras \cite{bk}, \cite{kashyap}. At the same time, 
the author of the present article proposed a different notion of Morita equivalence for dual operator algebras, called 
$\Delta $-equivalence. Two unital dual operator algebras $A$ and $B$ are $\Delta $-equivalent if there exist faithful 
normal representations $\alpha: A\rightarrow \alpha (A), \;\;\;\beta: B\rightarrow \beta (B)$ and a ternary ring of 
operators $M$ (i.e., a space satisfying $MM^*M\subseteq M$) such that $\alpha (A)=[M^*\beta (B)M]^{-w^*}$ and 
$\beta (B)=[M\alpha (A)M^*]^{-w^*}$ \cite{ele2}. In this case, we write $A\sim _\Delta B.$ An important property 
is that two algebras are $\Delta $-equivalent if and only if they are stably isomorphic, as was proved by Paulsen
 and the present author in \cite{ele4}. Subsequently, Paulsen, Todorov and the present author defined a Morita-type 
equivalence $\sim _\Delta $ for dual operator spaces \cite{ele5}. This equivalence also has the property of being 
equivalent with the notion of a stable isomorphism. 

In this paper, we define a weaker relation between dual operator 
algebras. We say that the dual operator algebra $B$ $\Delta $-embeds into the dual operator algebra $A$ if there exists 
a projection $p\in A$ such that $B\sim _\Delta pAp.$ In this case, we write $B\subset _\Delta A.$ We investigate the relation 
$\subset _\Delta $ between unital dual operator algebras, and we prove that it is a transitive relation. In the case of 
von Neumann algebras, it is a partial order relation. This means that it has the additional property that if $A, B$ are 
von Neumann algebras and $A\subset _\Delta B, B\subset _\Delta A,$ then $A\sim _\Delta B.$ We present a counterexample to 
demonstrate that this does not always hold in the case of non-self-adjoint algebras. In Section 2, we also present a 
characterisation of the relation $\subset _\Delta $ in the terms of reflexive lattices. 

In Section 3, we present an 
analogous theory defining the relation $\subset _\Delta $ for dual operator spaces. In this case, if $X, Y$ are dual 
operator spaces such that $Y\subset _\Delta X,$ then there exist projections $p$ and $q$ such that $pX\subseteq X,
 Xq\subseteq X$ and $Y\sim _\Delta pXq.$ We also define a weaker relation $\subset _{cb\Delta }.$ We say that 
$Y\subset _{cb\Delta }X$ if there exist $w^*$-continuous completely bounded isomorphisms $\phi: X\rightarrow \phi (X), 
\psi: Y\rightarrow \psi (Y) $ such that $\phi (X)\subset _\Delta \psi (Y).$ We present a theorem describing 
$\subset _{cb\Delta }$ in the terms of stable isomorphisms ( Theorem \ref{20}), and we investigate the problem 
of whether $\subset _\Delta $ is a transitive relation for dual operator spaces ( Theorem \ref{21}). 

In the following,
 we briefly describe the notions used in this paper. We refer the reader to the books \cite{bm}, \cite{dav}, \cite{er}, 
\cite{paul} and \cite{p} for further details. If $V$ is a linear space and $S\subseteq V,$ then by $[S]$ we denote the 
linear span of $S.$ If $H, K$ are Hilbert spaces, then we write $B(H, K)$ for the space of bounded operators from $H$ to $K.$
 We denote $B(H, H)$ as $B(H)$. If $L$ is a subset of $B(H),$ then we write $L^\prime $ for the commutant of $L,$ and 
$L^{\prime \prime }$ for $(L^\prime )^\prime. $ If $A$ is an operator algebra, then by $\Delta (A)$ we denote its diagonal 
$A\cap A^*.$ A ternary ring of operators $M,$ referred to as a TRO from this point, is a subspace of some $B(H, K)$ 
satisfying the following: $$m_1, m_2, m_3\in M\Rightarrow m_1m_2^*m_3\in M.$$ It is well known that in the case that $M$ 
is norm closed, it is equal to $[MM^*M]^{-\|\cdot\|}.$ If $X$ is a dual operator space and $I$ is a cardinal, then we write 
$M_I(X)$ for the set of $I\times I$ matrices whose finite submatrices have uniformly bounded norm. We underline that $M_I(X)$
 is also a dual operator space, and it is completely isometrically and $w^*$-homeomorphically isomorphic with $X\bar \otimes 
B(l^2(I)).$ Here, $\bar \otimes $ denotes the normal spatial tensor product. We say that two dual operator spaces $X$ and $Y$ 
are stably isomorphic if there exists a cardinal $I$ and a $w^*$-continuous completely isometric map from $M_I(X)$ onto $M_I(Y).$ 
If $\cl L\subseteq B(H)$ is a lattice, then we write $Alg(\cl L)$ for the algebra of operators $x\in B(H)$ satisfying 
$$l^\bot xl=0,\;\;\forall \;\;l\;\in \;\cl L.$$ If $A\subseteq B(H)$ is an algebra, then we write $Lat (A)$ for the 
lattice of projections $l\in B(H)$ satisfying $$l^\bot xl=0,\;\;\forall \;\;x\;\in \;A.$$ A lattice $\cl L$ is called 
reflexive if $$\cl L=Lat(Alg(\cl L)).$$ A reflexive algebra is an algebra of the form $Alg(\cl L)$, for some lattice 
$\cl L.$ An important example of a class of reflexive lattices is given by nests. A nest $\cl N\subseteq B(H)$ is a 
totally ordered set of projections containing the zero and identity operators, which is closed under arbitrary suprema 
and infima. The corresponding algebra $Alg(\cl N)$ is called a nest algebra. If $A\subseteq B(H)$ is a $w^*$-closed 
algebra and $I$ is cardinal, then we write $A^I$ for the algebra of operators $x\in B( H\otimes l^2(I) )$ satisfying
 $$x((\xi_i )_{i\in I})=(a(\xi _i))_{i\in I},\;\;\;\forall \;(\xi _i)_{i\in I}\;\in\; H\otimes l^2(I) $$ for some 
$a\in A.$ 

\section{Morita embeddings for dual operator algebras} 

We consider the following known theorem concerning 
von Neumann algebras. 

\begin{theorem}\label{1} Let $A, B$ be von Neumann algebras. Then, the following are equivalent: 

(i) There exist $w^*$-continuous, one-to-one, $*$-homomorphisms $$\alpha : A\rightarrow B(H),\;\; \beta : B\rightarrow B(K),$$ 
where $H, K$ are Hilbert spaces such that the commutants $\alpha (A)^\prime , \beta (B)^\prime $ are $*$-isomorphic. 

(ii) The algebras $A, B$ are weakly Morita equivalent in the sense of Rieffel. 

(iii) There exists a cardinal $I$ and a $*$-isomorphism 
from $M_I(A)$ onto $M_I(B).$ 
\end{theorem} 

\begin{definition}\label{2}\cite{ele1} Let $A, B$ be $w^*$-closed algebras acting on 
the Hilbert spaces $H$ and $K,$ respectively. We call these weakly TRO-equivalent if there exists a TRO $M\subseteq B(H,K)$ such
 that $$A=[M^*BM]^{-w^*}, B=[MAM^*] ^{-w^*}.$$ In this case, we write $A\sim _{TRO}B.$ \end{definition} 
The following defines 
our notion of weak Morita equivalence for dual operator algebras. 
\begin{definition}\label{3}\cite{ele2} Let $A, B$ be dual 
operator algebras. We call these weakly $\Delta $-equivalent if there exist $w^*$-continuous completely isometric homomorphisms 
$\alpha $ and $\beta,$ respectively, such that $\alpha (A)\sim _{TRO}\beta (B).$ In this case, we write $A\sim _{\Delta }B.$ 
\end{definition} 

The following theorem is a generalisation of Theorem \ref{1} to the setting of unital dual operator algebras:

 \begin{theorem}\label{4}Let $A, B$ be unital dual operator algebras. Then, the following statements are equivalent: 

(i) There
 exist reflexive lattices $\cl L_1$ and $\cl L_2,$ $w^*$-continuous completely isometric onto homomorphisms $\alpha :
 A\rightarrow \Alg{\cl L_1}, \;\;\;\beta : B\rightarrow \Alg{\cl L_2},$ and a $*$-isomorphism $\theta : \Delta (A)^\prime
 =\cl L_1^{\prime \prime } \rightarrow \Delta (B)^\prime =\cl L_2^{\prime \prime } $ such that $\theta (\cl L_1)=\cl L_2.$ 

(ii) The algebras $A, B$ are weakly $\Delta $-equivalent. 

(iii) There exists a cardinal $I$ and a $w^*$-continuous completely
 isometric homomorphism from $M_I(A)$ onto $M_I(B).$ \end{theorem} 

The previous theorem has been proved in various papers. 
In fact, if (i) holds, then by Theorem 3.3 in \cite{ele1} $\Alg{\cl L_1}\sim _{TRO} \Alg{\cl L_2},$ and thus $ A\sim _\Delta B. $ 
Conversely, if (ii) holds, then by Theorems 2.7 and 2.8 in \cite{ele3}, by choosing a $w^*$-continuous completely isometric 
homomorphism $\alpha : A\rightarrow \alpha(A) $ with reflexive range, there exists a $w^*$-continuous completely isometric 
homomorphism $\beta : B\rightarrow \beta (B)$, also with a reflexive range, such that $\alpha (A)\sim _{TRO}\beta (B).$ Thus,
 by Theorem 3.3 in \cite{ele1}, (i) holds. The equivalence of (ii) and (iii) constitutes the main result of \cite{ele4}. 

\begin{remark} In the remainder of this section, if $A$ is a unital dual operator algebra and $p\in A$ is a projection, then 
$pAp$ is also a dual operator algebra with unit $p.$ If $A$ is a $w^*$-closed unital algebra acting on the Hilbert space $H$ 
and $p\in A$ is a projection, then we identify $pAp$ with the algebra $pA|_{p(H)}\subseteq B(p(H)).$ \end{remark} 

\subsection{TRO-embeddings for dual operator algebras} 
\bigskip 
\begin{definition}\label{5} Let $A, B$ be unital $w^*$-closed 
algebras acting on the Hilbert spaces $H$ and $K$, respectively. We say that $B$ weakly TRO-embeds into $A$ if there exists a 
projection $p\in A$ such that $B\sim _{TRO}pAp.$ In this case, we write $B\subset _{TRO}A.$ \end{definition} 

\begin{remark} \label{6}
 The above definition is equivalent to the following statements. Let $A, B$ be unital $w^*$-closed algebras acting on the Hilbert 
spaces $H$ and $K,$ respectively. Then, 

(i) The algebra $B$ weakly TRO-embeds into $A$ if and only if there exists a TRO 
$M\subseteq B(H, K)$ such that $B=[M^*AM]^{-w^*}, \;\;\;MBM^*\subseteq A.$ 

(ii) The algebra $B$ weakly TRO-embeds into $A$ 
if and only if there exists a TRO $M\subseteq B(H, K)$ such that $B=[M^*AM]^{-w^*}, \;\;\;MM^*\subseteq A.$ \end{remark} 

\begin{remark} \label{66} If $A$ is a unital $w^*$-closed algebra and $p\in A$ is a projection, then $pAp\subset _{TRO}A.
$ For the proof, we can take the linear span of the element $p$ as a TRO. \end{remark} 

\begin{proposition} \label{7a}Suppose
 that $A, B, C$ are unital $w^*$-closed algebras acting on the Hilbert spaces $H, K$, and $L,$ respectively. If $C \subset _{TRO}B$ 
and $B\subset _{TRO}A,$ then $C\subset _{TRO}A.$ \end{proposition} 
\begin{proof} We may assume that there exist projections
 $p\in \Delta (B), q\in \Delta (A)$ such that $$C\sim _{TRO}pBp, \;\;B\sim _{TRO}qAq.$$ 
By Proposition 2.8 in \cite{ele1}, 
there exists a TRO $M$ such that $$B=[MqAqM^*]^{-w^*}, \;\;\;qAq=[M^*BM]^{-w^*},$$ $$\Delta (B)=[MM^*]^{-w^*}, q\Delta (A)q=[M^*M]^{-w^*}.$$ 
Define $N=pM.$ Then, we have that $$(pM)(pM)^*(pM)=pMM^*pM\subseteq pMM^*\Delta (B)M\subseteq p\Delta (B)M\subseteq pM.$$
 Thus, $N$ is a TRO. Then, we have that $$pBp=[pMqAqM^*p]^{-w^*}= [NqAqN^*]^{-w^*} = [N(N^*NqAqN^*N)N^*]^{-w^*}, $$ and 
therefore $$[N^*pBpN]^{-w^*}= [N^*NqAqN^*N]^{-w^*} .$$ Thus, $$pBp\sim _{TRO}[N^*NqAqN^*N]^{-w^*} .$$ Therefore, 
$$C\sim _{TRO}[N^*NqAqN^*N] ^{-w^*} .$$ We may assume that there exists a TRO $L$ such that 
\begin{equation}
\label{X} C=[LN^*NqAqN^*NL^* ] ^{-w^*} , \;\;[N^*NqAqN^*N]^{-w^*} = [L^*CL]^{-w^*}. \end{equation} 
We make the 
following observations: 
\begin{equation} \label{XX} [Nq\Delta (A)qN^*] ^{-w^*} =[pMM^*MM^*p] ^{-w^*} =[pMM^*p] ^{-w^*}
 =[NN^*] ^{-w^*} . \end{equation} 
Furthermore, $N^*N=M^*pM\subseteq [M^*M] ^{-w^*} =q\Delta (A)q.$ Thus, 
\begin{equation}
 \label{XXX} L^*L\subseteq \Delta ([N^*NqAqN^*N]^{-w^*})\subseteq q\Delta (A)q .\end{equation} 
Define $D=[LN^*Nq]^{-w^*}.$
 We shall prove that $D$ is a TRO. We have that $$(LN^*Nq)(qN^*NL^*)(LN^*Nq)=LN^*NqN^*(NL^*LN^*)Nq.$$ By (\ref{XXX}), it
 holds that $$NL^*LN^*\subseteq Nq\Delta (A)qN^*\subseteq NN^*.$$ Thus, $$DD^*D\subseteq [LN^*NqN^*NN^*Nq]^{-w^*}.$$ 
By (\ref{XX}), we have that $NqN^*\subseteq NN^*.$ Thus, $$DD^*D\subseteq [LN^*NN^*NN^*Nq]^{-w^*}=D.$$ Thus, $D$ is a
 TRO. By (\ref{X}), we have that $$C=[DAD^*]^{-w^*}.$$ Furthermore, $$D^*D=qN^*NL^*LN^*Nq.$$ 
By (\ref{XXX}), we have 
that $$D^*D\subseteq qN^*Nq\Delta (A)qN^*Nq.$$ In addition, by (\ref{XX}) we have that $$D^*D\subseteq qN^*NN^*Nq\subseteq 
qN^*Nq\subseteq q\Delta (A)q\subseteq A.$$ Thus, by Remark \ref{6} (ii), we have that $C\subset _{TRO}A.$ \end{proof}

\begin{remark}\label{667} In light of the above proposition, one could expect that the relation $\subset _{TRO}$ is a
 partial order relation in the class of unital $w^*$-closed operator algebras, if we identify those algebras that are 
TRO-equivalent. This means that $\subset _{TRO}$ has the additional property that $$A\subset _{TRO}B, \;\;\;B\subset _{TRO}A
 \Rightarrow A\sim_{TRO} B. $$ This is true in the case of von Neumann algebras, as we will show in Section 1.3. However, 
it fails in the case of non-self-adjoint algebras, as we prove in Section 1.4. \end{remark} 

The following Lemma will be 
useful. 

\begin{lemma}\label{**} Let $B$ be a $w^*$-closed unital operator algebra acting on the Hilbert space $H$, and 
let $q\in B$ be a projection. If $p$ is the projection onto $\overline{\Delta (B)(q(H))}$, then $p$ is a central projection
 for the algebra $\Delta (B)$, and $qBq\sim _{TRO}pBp.$ \end{lemma} 
\begin{proof}Clearly, $p$ is a central projection for 
$\Delta (B).$ We consider the TRO $$M=\Delta (B)q\subseteq B(q(H), p(H)).$$ We have that $$ \overline{M(q(H))}=p(H), \;\;
\overline{M^*(p(H))}=q(H),$$ and $$M^*pBpM\subseteq qBq,\;\;\;MqBqM^*\subseteq pBp.$$ Then, Proposition 2.1 in \cite{ele1} 
implies that $$[M^*pBpM]^{-w^*}=qBq,\;\;\;[MqBqM^*]^{-w^*}=pBp.$$ \end{proof} 

\subsection{$\Delta $-embeddings for dual 
operator algebras}
 \bigskip 
\begin{definition}\label{77} Let $A, B$ be dual operator algebras. We say that $B$ weakly 
$\Delta $-embeds into $A$ if there exist $w^*$-continuous completely isometric homomorphisms $\alpha : A\rightarrow 
\alpha (A), \;\;\;\beta : B\rightarrow \beta (B)$ such that $\beta(B)\subset _{TRO} \alpha (A).$ In this case, we write
 $B\subset _\Delta A.$ \end{definition}

 The following theorem is a generalisation of Theorem \ref{4}. 

\begin{theorem}\label{8}
 Let $A, B$ be unital dual operator algebras. Then, the following are equivalent: 

(i) There exist reflexive lattices 
$\cl L_1, \cl L_2$ acting on the Hilbert spaces $H$ and $K$, respectively, $w^* $-continuous completely isometric onto 
homomorphisms $$\alpha: A\rightarrow \Alg{\cl L_1},\;\; \beta: B\rightarrow \Alg{\cl L_2} ,$$ and an onto $w^*$-continuous 
$*$-homomorphism $$\theta : \Delta (A)^\prime =\cl L_1^{\prime \prime }\rightarrow \Delta (B)^\prime =\cl L_2^{\prime \prime }$$ 
such that $\theta (\cl L_1)=\cl L_2.$ 

(ii) $B\subset _\Delta A$. 

(iii) There exists a cardinal $I$, a projection $q\in A$, 
and a $w^*$-continuous completely isometric homomorphism from $M_I(B)$ onto $M_I(qAq).$ \end{theorem} 
\begin{proof} The 
equivalence of (ii) and (iii) is a consequence of Definition \ref{77} and Theorem \ref{4}.
 $$(i)\Rightarrow (ii)$$ 
It suffices 
to prove that $$\Alg{\cl L_2}\subset _{TRO}Alg{\cl L_1}.$$ Define the lattice $$\cl L=\{\left(\begin{array}{clr}l & 0 \\ 0 & 
\theta (l) \end{array}\right): \;\;l\in \cl L_1\}$$ and the spaces $$U=\{x: l^\bot x\theta (l)=0\;\;\forall \;l\in \cl L_1 \}, 
\;\;\;V=\{y: \theta (l)^\bot yl=0\;\;\forall \;l\in \cl L_1 \}.$$ We can easily prove that 
$$\Alg{\cl L}=\left(\begin{array}{clr}\Alg{\cl L_1} & U \\ V & \Alg{\cl L_2} \end{array}\right).$$
 Because the 
map $$\rho : \cl L_1^{\prime \prime }\rightarrow \cl L^{\prime \prime }, a\rightarrow \left(\begin{array}{clr}a & 0
 \\ 0 & \theta (a) \end{array}\right) $$ is a $*$-isomorphism such that $\theta (\cl L_1)=\cl L$, from Theorem 3.3 
in \cite{ele1} we have that $$\Alg{\cl L_1}\sim _{TRO}\Alg{\cl L}.$$ Define the TRO $$M=(0 \;\;\;\bb C).$$ Now, 
observe that $$M\Alg{\cl L}M^*=\Alg{\cl L_2}$$ and $$M^*M \subseteq \Alg{\cl L}.$$ Thus, $$\Alg{\cl L_2}\subset _{TRO}Alg{\cl L}.$$ 
Then, Proposition \ref{7a} implies that $$\Alg{\cl L_2}\subset _{TRO}Alg{\cl L_1}.$$ 

$$(iii)\Rightarrow (i)$$ 

Suppose that $M_I(B)$ 
and $M_I(qAq)$ are completely isometrically and $w^*$-homeomorphically isomorphic. Then, $B\sim _\Delta qAq.$ Every unital dual 
operator algebra has a $w^*$-completely isometric representation whose image is reflexive. Thus, we may assume that
 $\alpha : A\rightarrow B(H)$ is a $w^*$-continuous completely isometric homomorphism such that $\alpha (A)= \Alg{\cl L_1}$ 
for a reflexive lattice $\cl L_1.$ If $p=\alpha (q),$ then $$\alpha (qAq)= Alg(\cl L_1|_p(H)) .$$ Theorem 2.7 in \cite{ele3} 
implies that there exists a reflexive lattice $\cl L_2$ and a $w^*$-continuous completely isometric homomorphism 
$\beta : B\rightarrow Alg(\cl L_2) $ such that $$ Alg(\cl L_2) \sim _{TRO} Alg(\cl L_1|_p(H)) .$$ Thus, by Theorem 3.3 in 
\cite{ele1} there exists a $*$-isomorphism $\rho : (\cl L_1|_{p(H)})^{\prime \prime }\rightarrow \cl L_2^{\prime \prime } $ 
such that $\rho (\cl L_1|_{p(H)})=\cl L_2.$ If $\tau : \cl L_1^{\prime \prime }\rightarrow \cl L_1^{\prime \prime }|_{p(H)} :
 \;x\rightarrow x|_{p(H)}, $ then we write $\theta =\rho \circ \tau .$ This is the required map. \end{proof} 

\begin{theorem}\label{9}
 Let $A, B, C$ be unital dual operator algebras such that $$C\subset _{\Delta }B, \;\;\;B\subset _{\Delta }A.$$ 
Then, $C\subset _{\Delta }A.$ \end{theorem} 
\begin{proof}We may assume that there exist projections $p\in B, q\in A$ such that 
$$C\sim _{TRO}pBp, \;\;\;\beta (B)\sim _{TRO}qAq,$$ where $\beta : B\rightarrow \beta (B)$ is a $w^*$-continuous completely 
isometric homomorphism. From Theorem 2.7 in \cite{ele3}, we know that for the representation $\beta|_{pBp} : pBp\rightarrow 
\beta (p)\beta (B)\beta(p) $ there exists a $w^*$-continuous completely isometric homomorphism $\gamma: C\rightarrow \gamma(C) $
 such that $$\gamma (C)\sim _{TRO}\beta (p)\beta (B)\beta(p) .$$ By Proposition \ref{7a}, we have that $$\beta (p)\beta (B)\beta(p)
 \subset _{TRO}qAq.$$ Because $qAq\subset _{TRO}A$, we have that $\gamma (C)\subset _{TRO}A,$ and thus $C\subset _{\Delta }A.$ 
\end{proof}

 \begin{remark}\label{Y} In view of Theorem \ref{9}, one should expect  that weak $\Delta $-embedding is a partial 
order relation in the class of unital dual operator algebras if we identify those unital dual operator algebras that are weakly 
$\Delta $-equivalent. Thus, one should expect that $A\subset _\Delta B,\;\;B\subset _\Delta A\Rightarrow A\sim _\Delta B.$ 
In Section 1.4 we shall see that this is not true. However, in the case of von Neumann algebras, this is indeed true. 
For further details, see Section 2.3 below. \end{remark} 

\begin{example}\label{b} Let $A=\Alg{\cl N_1}, B=\Alg{\cl N_2},$ 
where $\cl N_1$ is a continuous nest, and $\cl N_2$ is a nest with at least one atom. We shall prove that it is impossible 
that $B\subset _\Delta A.$\end{example}
 \begin{proof} Suppose on the contrary that $B\subset _\Delta A.$ Thus, there exists a 
projection $p\in \Delta (A)$ such that $B\sim _\Delta pA|_{p}.$ Because $B$ and $pA|_p$ are nest algebras, it follows from 
Theorem 3.2 in \cite{ele3} that $B\sim _{TRO}pA|_p.$ Thus, by Theorem 3.3 in \cite{ele1} there exists a homeomorphism 
$\theta : \cl N_2\rightarrow \cl N_1|_p.$ This is impossible, because $\cl N_2$ contains an atom, and $\cl N_1|_p$ is a 
continuous nest. \end{proof} 

\subsection{The case of von Neumann algebras} 
\bigskip 

\begin{lemma}\label{sos} Let $A$ be a 
von Neumann algebra, and let $p,q$ be central projections of $A$ such that $p\leq q$ and $A\sim _{TRO}Ap.$ Then, 
$A\sim _{TRO}Aq.$ \end{lemma}
 \begin{proof} By Theorem 3.3 in \cite{ele1}, there exists a $*$-isomorphism $\theta : 
A^\prime \rightarrow A^\prime p.$ We need to prove that there exists a $*$-isomorphism $\rho : A^\prime \rightarrow A^\prime q.$ 
Suppose that $$e_0=Id_A, e_1=q, e_2=p, e_n=\theta (e_{n-2}), n=2,3,...$$ Clearly $(e_n)_n$ is a decreasing sequence of central 
projections. Observe that $$e_0=\left(\sum_{n=0}^\infty \oplus (e_{2n}-e_{2n+1})\oplus (e_{2n+1}-e_{2n+2})\right)\oplus \wedge _ne_n.$$
 Thus, the map $\rho : A^\prime \rightarrow A^\prime q$ sending $$a=\left(\sum_{n=0}^\infty \oplus a(e_{2n}-e_{2n+1})\oplus a
(e_{2n+1}-e_{2n+2})\right)\oplus (a\wedge _ne_n)$$ to $$\rho (a)=\left(\sum_{n=0}^\infty \oplus \theta (a)(e_{2n+2}-e_{2n+3})
\oplus a(e_{2n+1}-e_{2n+2})\right)\oplus (a\wedge _ne_n)$$ is a $*$-isomorphism . \end{proof} 

The above Lemma is based on the 
fact if $p, q$ are central projections of the von Neumann algebra $A$ such that $p\leq q$ and $A\cong Ap$, then $A\cong Aq,$ 
where $\cong $ is the $*$-isomorphism. We acknowledge that this was known to the authors of \cite{zar} (see the proof of Lemma 6.2.3).
 In this Lemma, an alternative proof to ours was provided. 

\begin{theorem}\label{xxxx} The relation $\subset _{TRO}$ is a partial 
order relation for von Neumann algebras, if we identify those von Neumann algebras that are TRO-equivalent. \end{theorem}
 \begin{proof}
 Let $A, B$ be von Neumann algebras. It suffices to prove the implication that $$A\subset _{TRO}B, \;\;\;B\subset _
{TRO}A \Rightarrow A\sim_{TRO} B. $$ Let $q_0\in B, p_0\in A$ be projections such that $$A\sim _{TRO}q_0Bq_0, \;\;B\sim _
{TRO}p_0Ap_0.$$ By Lemma \ref{**}, there exist central projections $q\in B, p\in A$ such that $$A\sim _{TRO}Bq, \;\;B\sim _{TRO}Ap.$$
 Thus, there exist $*$-isomorphisms $$\theta: A^\prime \rightarrow B^\prime q, \;\;\; \rho: B^\prime \rightarrow A^\prime p. $$ 
We can easily see that there exists a central projection $\hat p\in A$ such that $\hat p\leq p$ and $$\rho (B^\prime q)=A^\prime \hat p.$$ 
Therefore, we obtain a $*$-isomorphism from $A^\prime $ onto $A^\prime \hat p.$ Because $\hat p\leq p$ and $p, \hat p$ are central, Lemma 
\ref{sos} implies that there exists a $*$-isomorphism from $A^\prime $ onto $A^\prime p.$ Thus, $A\sim _{TRO}Ap,$ which implies that 
$A\sim _{TRO}B.$ \end{proof} 

\begin{theorem}\label{999} Let $A, B$ be von Neumann algebras. Then, the following are equivalent: 

(i) There exist $*$-isomorphisms $\alpha: A\rightarrow \alpha (A), \beta: B\rightarrow \beta (B) $ and a $w^*$-continuous onto 
$*$-homomorphism $\theta : \alpha (A)^\prime \rightarrow \beta (B)^\prime.$ 

(ii) $B\subset _\Delta A.$ 

(iii) There exists a 
cardinal $I$ and a $w^*$-continuous onto $*$-homomorphism $\rho : M_I(A)\rightarrow M_I(B).$ \end{theorem} 
\begin{proof}The 
equivalence of (i) and (ii) is a consequence of Theorem \ref{8}. 
$$(ii)\Rightarrow (iii)$$
 Suppose that $B\subset _\Delta A.$ 
By Theorem \ref{8}, we may assume that there exists a $w^*$-continuous onto $*$-homomorphism $\theta : A^\prime \rightarrow B^\prime.$ 
Suppose that $A^\prime p^\bot =Ker \theta $ for a projection $p$ in the centre of $A.$ We also assume that $A\subseteq B(H).$ Then, the map 
$$A^\prime|{p(H)}\rightarrow B^\prime: a|_{p(H)}\rightarrow \theta (a)$$ is a $*$-isomorphism. Because
 $$(A|_{p(H)})^\prime=A^\prime |_{p(H)},$$ Theorem \ref{1} implies that there exists a cardinal $I$ and a 
$w^*$-continuous onto $*$-isomorphism $$M_I(pAp)\rightarrow M_I(B).$$ 

$$(iii)\Rightarrow (ii)$$ 

Suppose that 
$\rho : M_I(A)\rightarrow M_I(B)$ is a $w^*$-continuous onto $*$-homomorphism. Let $p$ be a projection in the centre 
of $M_I(A)$ such that $$M_I(A)p^\bot =Ker \rho .$$ Because $Z(M_I(A))=Z(A)^I,$ where $Z(A)$ (resp. $Z(M_I(A))$) is the 
centre of $A$ (resp. $M_I(A)$), we may assume that $p=q^I$ for $q\in Z(A).$ Thus, the map $$M_I(Aq)\rightarrow M_I(B):
 (a_{i,j}q)\rightarrow \rho ((a_{i,j}))$$ is a $*$-isomorphism. Then, Theorem \ref{8} implies that $B\subset _\Delta A.$ 
\end{proof}

 \begin{theorem}\label{YY} The  weak $\Delta $-embedding is a partial order relation in the class of von Neumann
 algebras, if we identify those von Neumann algebras that are weakly Morita equivalent in the sense of Rieffel. 
\end{theorem}
 \begin{proof} \textbf{Claim:} Let $A$ be a von Neumann algebra, and let $r\in A$ be a projection such 
that $A\sim _\Delta rAr.$ If $q$ is a projection in $A$ such that $r\leq q$, then it also holds that $A\sim _\Delta qAq.$
 
\textbf{Proof of the claim:} There exists a cardinal $I$ such that the algebras $M_I(A)$ and $M_I(rAr)$ are $*$- isomorphic.
 We suppose that $M=M_I(A), \;\;\;N=M_I(qAq).$ Then, we have that $M\cong M_I(rAr)=r^INr^I$ and $N\cong q^IMq^I.$ 
By Lemma 6.2.3 in \cite{zar}, the von Neumann algebras $M$ and $N$ are stably isomorphic. Thus, $M\sim _\Delta N.$ 
However, $$A\sim _{TRO}M, \;\;\;qAq\sim _{TRO}N.$$ Therefore, $A\sim _\Delta qAq,$ and the proof of the claim is complete.

 To prove the theorem, it suffices to prove that if $A, B$ are von Neumann algebras such that $A\subset _\Delta B, 
\;\;B\subset _\Delta A$, then $A\sim _\Delta B.$ We may assume that there exist projections $p\in B, q\in A$ and 
$w^*$-continuous completely isometric homomorphisms $\alpha: A\rightarrow \alpha (A), \beta: B\rightarrow \beta(B) ,$ 
such that $$\alpha (A)\sim _{TRO}pBp, \;\;\;\beta (B)\sim _{TRO}qAq.$$ For the representation $\beta|_{pBp} : 
pBp\rightarrow \beta(p) \beta (B)\beta (p),$ there exists a $w^*$-continuous one-to-one $*$-homomorphism $\gamma :
\alpha (A)\rightarrow \gamma (\alpha(A)) $ such that $$\gamma (\alpha(A)) \sim _{TRO}\beta (p)\beta (B)\beta (p).$$
 By Proposition \ref{7a}, we have that $$\beta (p)\beta (B)\beta (p)\subset _{TRO}qAq.$$ Therefore, there exists a
 projection $r\leq q$ such that $$\gamma (\alpha (A))\sim _{TRO}rAr\Rightarrow A\sim _\Delta rAr.$$ The claim implies 
that $A\sim _\Delta qAq.$ However, $qAq\sim _\Delta B.$ Thus $A\sim _\Delta B.$ Thus, the proof is complete. \end{proof}

\begin{corollary} Let $A, B$ be von Neumann algebras, $I, J$ be cardinals, and $$\theta: M_I(A)\rightarrow M_I(B), \;\;\; 
\rho : M_J(B)\rightarrow M_J(A)$$ be onto $w^*$-continuous homomorphisms. Then, $A\sim _\Delta B.$ \end{corollary} 
\begin{proof} By Theorem \ref{999}, $A\subset _\Delta B$ and $B\subset _\Delta A.$ The conclusion then follows 
from the above theorem. \end{proof} 

\begin{corollary} Let $A, B$ be unital dual operator algebras such that 
$A\subset _\Delta B$ and $B\subset _\Delta A.$ Then, $\Delta (A)\sim _\Delta \Delta (B).$ \end{corollary}
 \begin{proof} 
We can easily see that $\Delta (A)\subset _\Delta \Delta (B)$ and $\Delta (B)\subset _\Delta \Delta (A).$ Now, we can
 apply the above theorem. \end{proof} 

\begin{example}\label{a} Let $A$ be a factor, and $B$ be a unital dual operator 
algebra such that $B\subset _\Delta A.$ Then, $B$ is a von Neumann algebra, and $B\sim _\Delta A.$ \end{example} 
\begin{proof} 
There exist a $*$-isomorphism $\alpha : A\rightarrow \alpha (A)$, a $w^*$- continuous completely isometric homomorphism 
$\beta : B\rightarrow \beta (B)$, and a TRO $M$ such that if $p$ is the projection onto $[MM^*]^{-w^*}$, then 
$$\beta (B)=[M^*\alpha (A)M]^{-w^*}, \;\;\;p\alpha (A)p=[M\beta (B)M^*]^{-w^*}.$$ Define $N=[\alpha (A)pM]^{-w^*}.$ 
Because $$MM^*p\alpha (A)\subseteq p\alpha (A)\subseteq \alpha (A),$$ it follows that $$pMM^*p\subseteq 
\alpha (A)\Rightarrow [\alpha (A)pMM^*p\alpha (A)]^{-w^*}=[NN^*]^{-w^*}$$ is an ideal of $\alpha (A).$ However,
 $\alpha (A)$ is a factor, and thus $\alpha (A)=[NN^*]^{-w^*}.$ On the other hand, $$[N^*N]^{-w^*}=[M^*p\alpha 
(A)\alpha (A)pM]^{-w^*}=[M^*\alpha (A)M]^{-w^*}=\beta (B).$$ Thus, $A$ and $B$ are weakly Morita equivalent in the 
sense of Rieffel. However, in the case of von Neumann algebras, Rieffel's Morita equivalence is the same as $\Delta $-equivalence.
 \end{proof} 

\subsection{A counterexample in non-self-adjoint operator algebras} 

Despite the situation for von Neumann algebras, 
we shall prove that if $A, B$ are unital non-self-adjoint dual operator algebras, it does not always hold that the implication
 \begin{equation}\label{****}A\subset _\Delta B, \;\;\;B\subset _\Delta A\Rightarrow A\sim _\Delta B.\end{equation} 

By Theorem 3.12 in \cite{ele3}, if $A, B$ are nest algebras, then $$A\sim _\Delta B\Leftrightarrow A\sim _{TRO}B.$$ 
Because for every nest algebra $B$ and every projection $p\in B$ the algebra $pBp$ is a nest algebra, we can conclude that 
$$A\subset _\Delta B\Leftrightarrow A\subset _{TRO}B.$$ Thus, in order to prove that (\ref{****}) does not hold, it 
suffices to find nest algebras $A$ and $B$ such that $A\subset _{TRO} B, \;\;B\subset _{TRO}A$ and $A$ is not TRO-equivalent 
to $B.$ Let $m$ be the Lebesgue measure on the Borel sets of the interval $[0,1].$ Suppose that $\bb Q$ is the set of 
rationals, and $Q^+$ (resp. $ Q^-$) is the projection onto $l^2(\bb Q\cap [0,t])$ (resp. $l^2(\bb Q\cap [0,t))$ ). 
Furthermore, let $N_t$ be the projection onto $L^2([0,t],m)$ for $0\leq t\leq 1$. We define the nest $$\cl N=\{ Q^+_t\oplus 
N_t, Q^-_t\oplus N_t,\;\;\; 0\leq t\leq 1\}.$$ By $A=Alg(\cl N),$ we denote the corresponding nest algebra acting on the Hilbert 
space $$H=l^2(\bb Q\cap [0,1])\oplus L^2([0,1],m).$$ The above nest appeared in Example 7.18 in \cite{dav}. Suppose that 
$f(t)=\frac{1}{2}, t\in [0,1],$ and define $$\cl M=\{ Q^+_{f(t)}\oplus N_{f(t)}, Q^-_{f(t)}\oplus N_{f(t)},\;\;\; 0\leq t\leq 1\}.$$
 We can define a unitary $$u_2: L^2([0,1]) \rightarrow L^2([0,\frac{1}{2}]) $$ such that $u_2(\chi _\Omega )=\sqrt{2}\chi _{f(\Omega )}$ 
where $\chi _\Omega $ is the characteristic function of the Borel set $\Omega .$ This unitary maps $N_t$ onto $N_{f(t)}$ in the
 sense that $$u_2N_tu_2^*=N_{f(t)}, \;\;0\leq t\leq 1.$$ Furthermore, the map $$ \{Q^+_t, Q^-_t: \;\;0\leq t\leq 1\} \longrightarrow 
\{Q^+_{f(t)}, Q^-_{f(t)}: \;\;0\leq t\leq 1\} $$ sending $Q^j_t$ onto $Q^j_{f(t)}$ for $j=+, -$ is a nest isomorphism. Because 
these nests are multiplicity free (they generate a maximal abelian self-adjoint algebra, referred to as an MASA from this point) 
and totally atomic, the above map extends as a $*$-isomorphism between the corresponding MASAs. Thus, there exists a unitary 
$$u_1: l^2(\bb Q\cap [0,1]) \rightarrow l^2(\bb Q\cap [0,\frac{1}{2}]) $$ such that $$ u_2Q^+_tu_2^*=Q^+_{f(t)}, \;\; u_2Q^-_tu_2^*=
Q^-_{f(t)}, 0\leq t\leq 1.$$ Therefore, the unitary $u=u_1\oplus u_2,$ implies a unitary equivalence between $\cl N$ and $\cl M.$ 

Let $s$ be the projection $$s: l^2(\bb Q\cap [0,1]) \rightarrow l^2(\bb Q\cap [0,\frac{1}{2}]) $$ and $r$ be the projection 
$$r: L^2([0,1],m ) \rightarrow L^2([0,\frac{1}{2}],m).$$ If $p=s\oplus r,$ then $p\in A$ and $pAp=Alg(\cl M).$ By the above 
arguments, $A$ and $pAp$ are unitarily equivalent, and thus they are TRO-equivalent. Suppose that $q_0$ is the projection 
$$q_0=\left(\begin{array}{clr} 0 & 0\\ 0 & I_{L^2([0,1])-r}\end{array}\right)\in A$$ and $q=p+q_0.$ Because $p\leq q\leq Id_A$,
 we have that $$pAp\subset_{TRO} qAq \subset_{TRO} A.$$ However, $A\sim _{TRO}pAp.$ This implies that $A\subset _{TRO}qAq.$ 
Thus, if (\ref{****}) holds, then we should have that $$A\sim _{TRO}qAq.$$ Suppose that $\cl L$ is the nest $Lat (qAq).$ 
By Theorem 3.3 in \cite{ele1}, there exists a $*$-isomorphism $$\theta : \Delta (A) ^\prime \rightarrow ( \Delta (A)|_{q(H)} )
^\prime $$ such that $\theta (\cl N)=\cl L.$ However, the algebras $\Delta (A), \Delta (A)|_{q(H)} $ are MASAs. Therefore, 
there exists a unitary $w: q(H)\rightarrow H$ such that $$\theta (x)=w^*xw,\;\;\forall x\;\in \;\Delta (A)=\Delta (A)^\prime .$$ 
We have that $$A=wqAqw^*.$$ We can easily see that $\cl L=\cl L_1\cup \cl L_2,$ where $$\cl L_1=\{Q^+_t\oplus N_t, Q^-_t\oplus N_t,
\;\;\; 0\leq t\leq \frac{1}{2}\}$$ and $$\cl L_2=\{ Q^+_{ \frac{1}{2}}\oplus N_t, \;\;\; \frac{1}{2}\leq t \leq 1\}.$$
 Observe that $L_1\leq L_0\leq L_2$ for all $L_i\in \cl L_i, i=1,2$ where $L_0=Q^+_ {\frac{1}{2}} \oplus N_{\frac{1}{2}} .$ 
If $$M_0=wL_0w^*, \;\;\cl N_1=w\cl L_1w^*, \;\;\cl N_2=w\cl L_2w^*,$$ then $$M_1\leq M_0\leq M_2$$ for all $M_i\in \cl N_i, i=1,2.$ 
Suppose that 
$M_0=Q^+_ {t_0} \oplus N_{t_0} .$ Then, $$\cl N_2=\{ Q^+_t\oplus N_t, Q^-_t\oplus N_t,\;\;\; t_0\leq t\leq 1\}.$$ If 
$$ \hat{ \cl L_2} =\{(Q^+_{ \frac{1}{2}}\oplus N_t)-L_0: \;\;\frac{1}{2}\leq t\leq 1\},$$ we can consider $\hat{ \cl L_2}$ 
to be a nest acting on $L^2([\frac{1}{2},1],m).$ Furthermore, if $$\hat{ \cl N_2}=\{ (Q^+_{t}\oplus N_t)-M_0, (Q^-_{t}\oplus N_t)-M_0,
 \;\;t_0\leq t\leq 1\},$$ then $\hat{ \cl L_2}$ and $\hat{ \cl N_2}$ are isomorphic nests. However, this is impossible, because 
$\hat{ \cl L_2} $ is a continuous nest and $\hat{ \cl N_2} $ is a nest with atoms. This contradiction shows that $A$ and $qAq$ 
are not TRO-equivalent.

 \begin{remark} Let $A$ and $q$ be as above. As we have seen, $qAq\subset _{\Delta }A, \;\;A\subset _\Delta qAq$
 but $A$ and $qAq$ are not $\Delta $-equivalent. We can prove further that they are not Morita equivalent even in the sense of Blecher
 and Kashyap \cite{bk}, \cite{kashyap}. If they were, then by \cite{elenest} $\cl N$ and $\cl L$ would be isomorphic as nests. However,
 we can see that this is impossible by applying the same arguments as above. \end{remark} 

\section{Morita embeddings for dual operator
 spaces} 

Definition \ref{2} can be adapted to the setting of dual operator spaces as follows. 

\begin{definition} \label{10} \cite{ele5}
 Let $H_1, H_2, K_1, K_2$ be Hilbert spaces, and let $$X\subseteq B(H_1, H_2), Y\subseteq B(K_1, K_2)$$ be $w^*$-closed spaces. 
We call these weakly TRO-equivalent if there exist TROs $M_i\subseteq B(H_i, K_i), i=1,2$ such that $$X=[M_2^*YM_1]^{-w^*}, \;\;\;Y=
[M_2XM_1^*]^{-w^*}.$$ In this case, we write $X\sim _{TRO}Y.$ \end{definition} 

\begin{remark}\label{ivan} If $W_1, W_2$ are Hilbert 
spaces and $Z$ is a subspace of $B(W_1, W_2),$ then we call it nondegenerate if $ \overline{Z(W_1)}=W_2 ,\;\; \overline{Z^*(W_2)}=W_1. $ 
If $H_1, H_2, K_1,\\ K_2, X, Y$ are as in the above definition, $p_2$ (resp. $q_2$) is the projection onto $\overline{X(H_1)}$ 
(resp. $\overline{Y(K_1)}$), and $p_1$ (resp. $q_1$) is the projection onto $\overline{X^*(H_2)}$ (resp. $\overline{Y^*(K_2)}$), 
then the spaces $p_2X|_{p_1(H_1)}, \;q_2Y|_{q_1(K_1)}$ are nondegenerate, and also weakly TRO-equivalent. This can be concluded from 
Proposition 2.2 in \cite{ele5}. \end{remark}

 The following defines our notion of weak Morita equivalence for dual operator spaces.
 
\begin{definition}\label{11} \cite{ele5} Let $X, Y$ be dual operator spaces. We call these weakly $\Delta $-equivalent if there exist 
$w^*$-continuous completely isometric maps $\phi, \psi,$ respectively, such that $\phi (X)\sim _{TRO}\psi (Y).$ In this case, we 
write $X\sim _\Delta Y.$ \end{definition} 

The following theorem constitutes the main result of \cite{ele5}. 

\begin{theorem}\label{12}
Let $X, Y$ be dual operator spaces. Then, the following are equivalent: 

(i) $X\sim _\Delta Y.$ 

(ii) There exists a cardinal $I$ and
 a $w^*$-continuous completely isometric map from $M_I(X)$ onto $M_I(Y).$ \end{theorem}

 \begin{remark} Throughout Section 3, we shall
 employ the following notation. If $H_1, H_2$ are Hilbert spaces, $X\subseteq B(H_1, H_2)$ is a $w^*$-closed subspace, 
and $q\in B(H_1), \;p\in B(H_2)$ are projections such that 
$pX\subseteq X,\;Xq\subseteq X$, then by $pXq$ we denote the space $\{pxq: x\in X\}\subseteq B(H_1, H_2).$ This space is 
$w^*$-closed 
and completely isometrically and $w^*$-homeomorphically isomorphic with the space $pX|_{q(H_1)}\subseteq B(q(H_1), p(H_2)).$ 
\end{remark} 

\subsection{TRO-embeddings for dual operator spaces} 

\begin{definition} \label{13} Let $H_1, H_2, K_1, K_2$ be 
Hilbert spaces, and let $X\subseteq B(H_1, H_2), \\Y\subseteq B(K_1, K_2)$ be $w^*$-closed spaces. We say that $Y$ weakly TRO 
embeds into $X$ if there exist TROs $M_1\subseteq B(H_1, K_1)$ and $M_2\subseteq B(H_2, K_2)$ such that $Y=[M_2XM_1^*]^{-w^*}, 
\;\;M_2^*YM_1\subseteq X$ and $M_2^*M_2X\subseteq X, \;\;XM_1^*M_1\subseteq X.$ In this case, we write $Y\subset _{TRO}X.$ 
\end{definition} 

\begin{remark}\label{soter} We can easily see that if $Y\subset _{TRO}X,$ then there exist projections $p, q$ 
such that $pX\subseteq X, \;\;Xq\subseteq X$ and $Y\sim _{TRO}pXq.$ \end{remark} 

\begin{examples}\label{examp}

\em{ (i) If 
$X\sim _{TRO}Y$, then clearly $X\subset _{TRO}Y$ and $Y\subset _{TRO}X.$ 

(ii) If $K_i, W_i, i=1,2$ are Hilbert spaces, 
$$Y\subseteq B(K_1, K_2),\;\;Z\subseteq B(W_1, W_2)$$ are $w^*$-closed spaces, and $$X=Y\oplus Z\subseteq B(K_1\oplus W_1,
 K_2\oplus W_2),$$ then $Y\subset _{TRO}X.$ For the proof, we apply the TROs $M_1=( \bb CI_{K_1}, 0), \;\;\;M_2=( \bb CI_{K_2}, 0).$ 

(iii) If $X\subseteq B(H_1, H_2)$ is a $w^*$-closed operator space and $p\in B(H_2), q\in B(H_1)$ are projections such that 
$pX\subseteq X, \;\;Xq\subseteq X,$ then $pXq\subset _{TRO}X.$ 

(iv) A generalisation of $W^*$-modules over von Neumann algebras 
is given by the projectively $w^*$-rigged modules over unital dual operator algebras. See \cite{blekra} for more details. 
Given a unital dual operator algebra $A,$ a projectively $w^*$-rigged module over $A$ is a dual operator space $Z$ that is 
completely isometrically and $w^*$-homeomorphically isomorphic to a space $Y=[MA]^{-w^*},$ where $M$ is a TRO satisfying 
$M^*M\subseteq A.$ Observe that $Y=[MA\bb C]^{-w^*},\;\;\;M^*Y\bb C\subseteq A$ and $M^*MA\subseteq A, \;\;A\bb C\subseteq A.$ 
Thus, $Y\subset _{TRO}A.$ Therefore, for every projectively $w^*$-rigged module $Z$ over a unital dual operator algebra $A,$ 
we have that $Z\subset _\Delta A.$ Here, $\subset _\Delta $ is the relation defined in Definition \ref{7} below.} \end{examples}

 \begin{proposition}\label{14b} Let $X, Y, Z$ be $w^*$-closed operator spaces. If $$Z\subset _{TRO}Y, \;\;\;Y\subset _{TRO}X, $$ 
then there exist projections $p, q$ such that $pX\subseteq X, \;\;Xq\subseteq X$ and $Z\subset _{TRO}pXq.$ \end{proposition} 
\begin{proof} There exist projections $p,q,r,s$ such that $$pX\subseteq X, \;Xq\subseteq X,\;rY\subseteq Y,\;Ys\subseteq Y$$ 
and TROs $M_i, N_i, i=1,2$ such that $$Y=[M_2pXqM_1^* ]^{-w^*} ,\;\;pXq=[M_2^*YM_1]^{-w^*} ,$$
$$Z=[N_2rYsN_1^*]^{-w^*},\;rYs=
[N_2^*ZN_1]^{-w^*}.$$ We may assume that $$M_2p=M_2,\;M_1q=M_1,\;N_2r=N_2,\;N_1s=N_1. $$ Suppose that $D_i$ is the $W^*$-algebra 
generated by the set $$\{M_iM_i^*\}\cup \{N_i^*N_i\}, i=1,2.$$ Define $$L_i=[N_iD_iM_i]^{-w^*}, i=1,2.$$ Because $M_1M_1^*\subseteq 
D_1, \;\;N_1^*N_1\subseteq D_1,$ it follows that $$N_1D_1M_1M_1^*D_1N_1^*N_1D_1M_1\subseteq N_1D_1M_1,$$ and thus $$L_1L_1^*L_1
\subseteq L_1.$$ Therefore, $L_1,$ and similarly $L_2,$ are TROs. Now, we have that $$[L_2pXqL_1^* ]^{-w^*} =[N_2D_2M_2pXqM_1^*D_1N_1^*]^
{-w^*} =[N_2D_2YD_1N_1^*]^{-w^*} .$$ Because $$[M_2M_2^*Y]^{-w^*} =Y=[YM_1M_1^*]^{-w^*} ,\;N_2^*N_2Y\subseteq Y, \;YN_1^*N_1\subseteq Y,$$ 
we have that $D_2Y=Y=YD_1.$ 
Thus, \begin{equation}\label{sav1} [L_2pXqL_1^* ]^{-w^*} =[N_2YN_1^*]^{-w^*} =[N_2rYsN_1^*]^{-w^*} =Z. 
\end{equation} 
Furthermore, $$L_2^*ZL_1\subseteq [M_2^*D_2N_2^*ZN_1D_1M_1 ]^{-w^*} =[M_2^*D_2rYsD_1M_1]^{-w^*} \subseteq $$ 
$$[M_2^*D_2YD_1M_1]^{-w^*} =[M_2^*YM_1]^{-w^*} .$$ Thus, 
\begin{equation}\label{sav2} L_2^*ZL_1\subseteq pXq. \end{equation} 
On the other hand, $$L_2^*L_2pXq\subseteq [M_2^*D_2N_2^*N_2D_2M_2pXq ]^{-w^*}\subseteq [M_2^*D_2M_2pXq]^{-w^*}=$$ 
$$[M_2^*D_2M_2M_2^*YM_1]^{-w^*}\subseteq [M_2^*YM_1]^{-w^*}.$$ Thus, $L_2^*L_2pXq\subseteq pXq$, and similarly 
$pXqL_1^*L_1\subseteq pXq.$ Therefore, the relations (\ref{sav1}) and (\ref{sav2}) imply that $Z\subset _{TRO}pXq.$ 
\end{proof}

 \begin{remark}\label{350000} From the above proof, we isolate the fact that if $Z\subset _{TRO}Y$ and $Y\sim _{TRO}X,$ 
then $Z\subset _{TRO}X.$ \end{remark}

 \subsection{$\Delta $-embeddings for dual operator spaces} 

\begin{definition}\label{7} 
Let $X, Y$ be dual operator spaces. We say that $Y$ weakly $\Delta $-embeds into $X$ if there exist $w^*$-continuous completely 
isometric maps $$\phi : X\rightarrow \phi (X), \;\;\;\psi :Y \rightarrow \psi (Y)$$ such that $\psi (Y)\subset _{TRO} \phi (X).$ 
In this case, we write $Y\subset _\Delta X.$ \end{definition} 

\begin{definition}\label{16a} Let $X, Y$ be dual operator spaces. A map 
$\phi : X\rightarrow Y$ that is one-to-one, $w^*$-continuous, and completely bounded with a completely bounded inverse is called a 
$w^*$-c.b. isomorphism, and the spaces $X, Y$ are called $w^*$-c.b. isomorphic. Under the above assumptions, the map $\phi ^{-1}$ is 
also $w^*$-continuous. \end{definition} 

\begin{definition}\label{16} Let $X, Y$ be dual operator spaces. We call these c.b. 
$\Delta $-equivalent if there exist $w^*$- c.b. isomorphisms $\phi: X\rightarrow \phi(X), \;\; \psi: Y\rightarrow \psi (Y)$ such 
that $\phi(X)\sim _{TRO} \psi (Y).$ In this case, we write $X\sim _{cb\Delta }Y.$ \end{definition} 

\begin{definition}\label{17}Let 
$X, Y$ be dual operator spaces. We say that $Y$ c.b. $\Delta $-embeds into $X$ if there exist $w^*$-c.b. isomorphisms 
$\phi: X\rightarrow \phi (X), \psi: Y\rightarrow \psi (Y) $ such that $\psi(Y) \subset _{TRO}\phi (X).$ In this case, we 
write $Y\subset _{cb\Delta }X.$ \end{definition} 

\begin{remark}\label{18} Observe the following: 

(i) $X\sim _\Delta Y\Rightarrow X\sim _{cb\Delta }Y$ 

(ii) $X\subset _\Delta Y\Rightarrow X\subset _{cb\Delta }Y$ \end{remark}

 In what follows, if $X$ is a dual operator space, then $M_l(X)$ (resp. $M_r(X)$) denotes the algebra of left (resp. right) 
multipliers of $X.$ In this case, $A_l(X)=\Delta (M_l(X)),$ (resp. $A_r(X)=\Delta (M_r(X))$) is a von Neumann algebra \cite{bm}.

 \begin{lemma}\label{xx} Suppose that $Z, Y$ are $w^*$-closed operator spaces satisfying $Z\sim _{TRO}Y,$ $H_1, H_2$ are Hilbert 
spaces such that $$A_l(Y)\subseteq B(H_2), \;\;\;A_r(Y)\subseteq B(H_1),$$ and $\psi: Y\rightarrow B(H_1, H_2)$ is a $w^*$-continuous 
complete isometry such that $$A_l(Y)\psi (Y)A_r(Y)\subseteq \psi (Y).$$ Then, there exists a $w^*$-continuous complete isometry 
$\zeta: Z\rightarrow \zeta (Z)$ such that $\zeta (Z) \sim _{TRO} \psi (Y).$ \end{lemma}
 \begin{proof} Assume that $M_1, M_2$ are 
TROs such that $$Z=[M_2^*YM_1]^{-w^*}, \;\;\;Y=[M_2ZM_1^*]^{-w^*}.$$ By Remark \ref{ivan}, we may assume that $Z$ and $Y$ are 
nondegenerate spaces. We denote $$A=[M_2^*M_2]^{-w^*}, \;\; B=[M_1^*M_1] ^{-w^*},\;\; C=[M_2M_2^*] ^{-w^*},\;\; D=[M_1M_1^*] ^{-w^*} . $$ 
The algebras $$\Omega (Z)= \left(\begin{array}{clr} A & Z \\ 0 & B \end{array}\right) , \;\;\;\Omega (Y)= \left(\begin{array}{clr}
 C & Y \\ 0 & D \end{array}\right) $$ are weakly TRO-equivalent as algebras. Indeed, $$ \Omega (Z)=[M^*\Omega (Y)M]^{-w^*} ,
 \;\;\;\Omega (Y)=[M\Omega (Z)M^*]^{-w^*}, $$ where $M$ is the TRO $M_2\oplus M_1.$ If $c\in C,$ then define $$\gamma (c): 
\psi (Y)\rightarrow \psi(Y),\;\;\;\gamma (c)\psi (y)=\psi (cy). $$ We can easily see that $\gamma (c)\in A_l(Y)$ and $\|\gamma (c)\|\leq 1.$
 Thus, $\gamma :C\rightarrow A_l(Y)$ is a contractive homomorphism and hence a $*$-homomorphism. If $\gamma (c)=0,$ then $cY=0.$
 Because $Y$ is nondegenerate, we conclude that $c=0.$ Thus, $\gamma $ is a one-to-one $*$-homomorphism. Similarly,
 there exists a one-to-one $*$-homomorphism $\delta : D\rightarrow A_r(Y)$ such that $\psi (y)\delta (d)=\psi (yd),\;\;\forall \;y.$
 The map $\pi : \Omega (Y)\rightarrow \pi (\Omega(Y)) $, given by $$\pi \left(\left(\begin{array}{clr} c & y \\ 0 & d \end{array}\right) 
\right)= \left(\begin{array}{clr} \gamma (c)& \psi (y) \\ 0 & \delta (d)\end {array}\right), $$ is a $w^*$-continuous completely 
isometric homomorphism. This can be shown by applying 3.6.1 in \cite{bm}. By Theorem 2.7 in \cite{ele3}, there exists a $w^*$-continuous 
completely isometric homomorphism $\rho : \Omega (Z)\rightarrow \rho (\Omega (Z))$ and a TRO $N$ such that 
$$ \rho (\Omega (Z))=[N^*\pi (\Omega (Y))N]^{w^*} ,\;\;\;\pi (\Omega (Y))=[N\rho (\Omega (Z))N^*]^{w^*}. $$ As in the
 discussion concerning the map $\Phi$ below Theorem 2.5 in \cite{ele5}, the map $\rho $ is given by 
$$\rho \left(\left(\begin{array}{clr} a & z \\ 0 & b \end{array}\right) \right)= \left(\begin{array}{clr} \alpha (a) & 
\zeta (z) \\ 0 & \beta (b)\end {array}\right), $$ where $\alpha : A\rightarrow \alpha (A), \;\;\zeta : Z\rightarrow \zeta (Z), 
\;\;\beta : B\rightarrow \beta (B)$ are completely isometric maps. By Lemma 2.8 in \cite{ele5}, the TRO $N$ is of the form 
$N=N_2\oplus N_1.$ Thus, $$ \zeta (Z)=[N_2^*\psi (Y)N_1]^{-w^*}, \;\;\;\psi (Y)=[N_2\zeta (Z)N_1^*]^{-w^*}. $$ \end{proof}

\begin{lemma}\label{19} Let $Z, \Omega , X$ be dual operator spaces. We assume that $Z\sim _{TRO}\Omega $, and that $\psi _0: 
\Omega \rightarrow \psi_0( \Omega )$ is a $w^*$-continuous complete isometry such that $\psi _0(\Omega ) \subset _{TRO}X.$ Then, 
there exists a $w^*$-c.b. isomorphism $\hat \phi : X\rightarrow \hat \phi (X)$ and a $w^*$-continuous complete isometry $\zeta : 
Z\rightarrow \zeta (Z)$ such that $\zeta (Z)\subset _{TRO}\hat \phi (X).$ Thus, $Z\subset _{cb\Delta }X.$ \end{lemma} 
\begin{proof}
 Suppose that $$A_l(\Omega )\subseteq B(H_2), \;\;\;A_r(\Omega )\subseteq B(H_1)$$ and $\psi : \Omega \rightarrow B(H_1, H_2)$ is a
 $w^*$-continuous complete isometry such that $$A_l(\Omega )\psi(\Omega ) A_r(\Omega )\subseteq \psi (\Omega ).$$ By Lemma \ref{xx},
 there exists a $w^*$-continuous complete isometry $\zeta : Z\rightarrow \zeta (Z)$ such that $$\zeta (Z)\sim _{TRO}\psi (\Omega ).$$ 
We assume that $p,q$ are projections such that $pX\subseteq X, Xq\subseteq X$ and $$\psi _0(\Omega ) \sim _{TRO}pXq.$$ 
Again by Lemma \ref{xx}, there exists a $w^*$-continuous complete isometry $\phi : pXq\rightarrow \phi (pXq)$ such that
 $\psi (\Omega)\sim _{TRO} \phi (pXq).$ Define $$\hat \phi (x)= \left(\begin{array}{clr} \phi (pxq) & 0 & 0 \\ 0 & p^\bot xq & 0 
\\ 0 & 0 & xq^\bot \end{array}\right) ,$$ for all $x\in X.$ Observe that $\hat \phi $ is a $w^*$-continuous completely bounded 
and one-to-one map. If $\hat \phi ^\infty $ is the $\infty \times \infty $ amplification of $\hat \phi $, then $\hat \phi ^\infty $ 
has a closed range. Thus, by the open mapping theorem, $\hat \phi ^\infty $ has a bounded inverse. Therefore, $\hat \phi ^{-1}$ 
is completely bounded. We have that $\zeta (Z)\sim _{TRO}\phi (pXq).$ Thus, there exist TROs $M_1, M_2$ such that 
$$\zeta (Z)=[M_2\phi (pXq)M_1^*]^{-w^*}, \;\;\;\phi (pXq)=[M_2^*\zeta (Z)M_1]^{-w^*}.$$ Define the TROs 
$$N_i=\left(\begin{array}{clr} M_i & 0 &0\end{array}\right), \;\;i=1,2.$$ We can see that $$[N_2\hat \phi (X)N_1^*]^{-w^*}
=\zeta (Z),\;\;\;N_2^*\zeta (Z)N_1\subseteq \hat \phi (X)$$ and $$ N_2^*N_2\hat \phi (X)\subseteq \hat \phi (X), \;\;
\hat \phi (X)N_1^*N_1\subseteq \hat \phi (X). $$ Thus, $\zeta (Z)\subset _{TRO}\hat \phi (X).$ \end{proof} 

\begin{theorem}\label{20} 
Let $X, Y$ be dual operator spaces. Then, the following are equivalent. 

(i) $Y\subset _{cb\Delta }X.$ 

(ii) There exist $w^*$-c.b. 
isomorphisms $\psi : Y\rightarrow \psi (Y), \;\;\;\phi: X\rightarrow \phi(X); $ projections $p,q$ such that $p\phi (X)\subseteq 
\phi(X), \;\;\; \phi (X)q\subseteq \phi (X)$; a cardinal $I$; and a completely isometric $w^*$-continuous onto map $$\pi : 
M_I(\psi (Y))\rightarrow M_I(p\phi (X)q).$$ \end{theorem} 
\begin{proof}
 $$(i)\Rightarrow (ii)$$

 By definition, there exist 
$w^*$-c.b. isomorphisms $\psi : Y\rightarrow \psi (Y), \;\;\;\phi: X\rightarrow \phi(X) $ such that $$\psi (Y)\subset _{TRO}
\phi (X).$$ There exist projections $p,q$ such that $p\phi (X)\subseteq \phi(X), \;\;\; \phi (X)q\subseteq \phi (X)$ and 
$\psi(Y)\sim _{TRO} p\phi (X)q.$ By Theorem \ref{12}, there exists a cardinal $I$ and a completely isometric $w^*$-continuous 
onto map $$\pi : M_I( \psi (Y) )\rightarrow M_I( p\phi (X)q ).$$ 

$$(ii)\Rightarrow (i)$$ 

Define $$\hat \phi (x)= 
\left(\begin{array}{clr} p\phi (x)q & 0 & 0 \\ 0 & p^\bot \phi (x)q & 0 \\ 0 & 0 & \phi (x)q^\bot \end{array}\right), $$ 
for all $x\in X.$ As in the proof of Lemma \ref{19}, we can see that $\hat \phi $ is a $w^*$-c.b. isomorphism. By Example 
\ref{examp} (ii), we have that $$p\phi (X)q\subset _{TRO}\hat \phi (X).$$ By Theorem \ref{12}, it holds that 
$\psi (Y)\sim _\Delta p\phi (X)q .$ Thus, there exist completely isometric $w^*$-continuous maps $$\mu : \psi (Y)\rightarrow 
\mu (\psi (Y)),\;\;\;\chi : p\phi (X)q\rightarrow \chi (p\phi (X)q)$$ such that $\mu (\psi (Y))\sim _{TRO} \chi ( p\phi (X)q ) .$ 
Now, apply Lemma \ref{19} for $$Z= \mu (\psi (Y)) , \;\;\;\Omega =\chi ( p\phi (X)q ) ,\;\;\;\psi _0=\chi ^{-1}: \Omega \rightarrow
 p\phi (X)q.$$ We have that $$\psi _0(\Omega )\subset _{TRO}\hat \phi (X).$$ We conclude that $$\mu (\psi (Y)) \subset _{cb\Delta } 
\hat \phi (X)\Rightarrow Y\subset _{cb\Delta }X.$$ \end{proof} 

\begin{lemma}\label{soson} Suppose that 
$Z\subseteq B(W_1, W_2),\;Y\subseteq B(K_1, K_2)$ are $w^*$-closed spaces such that $Z\subset _{TRO}Y.$ We also assume 
that $p_2$ is the projection onto $\overline{Y(K_1)},$ and $p_1$ is the projection onto $\overline{Y^*(K_2)}.$ Thus, 
$Y_0=p_2Y|_{p_1(K_1)}$ is a nondegenerate space into $B(p_1(K_1), p_2(K_2)).$ We shall prove that $Z\subset _{TRO}Y_0.$ 
\end{lemma} 
\begin{proof} By definition, there exist TROs $N_i\subseteq B(W_i, K_i), \;i=1,2$ such that $$Z=[N_2^*YN_1]^{-w^*},
\;N_2ZN_1^*\subseteq Y,\;N_2N_2^*Y\subseteq Y,\;YN_1N_1^*\subseteq Y.$$ Define $M_i=p_iN_i\subseteq B(W_i, p_i(K_i)),\;i=1,2.$ 
We have that $$p_2y=y\Rightarrow p_2m_1m_2^*y=m_1m_2^*y, \;\;\forall \;y\;\in \;Y,\;\;m_1, m_2\in N_2.$$ Thus, $$ p_2N_2N_2^*p_2 
=N_2N_2^*p_2\Rightarrow p_2N_2N_2^*p_2 =p_2N_2N_2^*.$$ The above relation implies that $$M_2M_2^*M_2=p_2N_2N_2^*p_2N_2=p_2N_2N_2^*N_2
\subseteq p_2N_2=M_2.$$ Therefore, $M_2$ is a TRO. Similarly, $M_1$ is also a TRO. Now, we have that $$[M_2^*Y_0M_1]^{-w^*}=
[N_2^*p_2Yp_1N_1]^{-w^*}=[N_2^*YN_1]^{-w^*}=Z$$ and $$M_2ZM_1^*=p_2N_2ZN_1^*p_1\subseteq p_2Yp_1=Y_0.$$ Furthermore, 
$$M_2M_2^*Y_0=p_2N_2N_2^*p_2Yp_1=p_2N_2N_2^*Yp_1\subseteq p_2Yp_1=Y_0.$$ Therefore, $Y_0$ is a nondegenerate subspace 
of $B(p_1(K_1), p_2(K_2))$, and $Z\subset _{TRO}Y_0.$ \end{proof}

 The Lemma below is weaker than Lemma \ref{xx}.

\begin{lemma}\label{270000} Suppose that $Z, Y$ are $w^*$-closed operator spaces satisfying $Z\subset _{TRO}Y$, $H_1, H_2$
 are Hilbert spaces such that $$A_l(Y)\subseteq B(H_2), \;\;\;A_r(Y)\subseteq B(H_1),$$ and $\psi: Y\rightarrow B(H_1, H_2)$ 
is a $w^*$-continuous complete isometry such that $$A_l(Y)\psi (Y)A_r(Y)\subseteq \psi (Y).$$ Then, there exists a $w^*$-continuous 
complete isometry $\zeta: Z\rightarrow \zeta (Z)$ such that $\zeta (Z) \subset _{TRO} \psi (Y).$ \end{lemma}

 \begin{proof}
 Assume that $M_1, M_2$ are TROs such that $$Z=[M_2^*YM_1] ^{-w^*} ,\;\; M_2ZM_1^* \subseteq Y, \;\;\;M_2M_2^*Y\subseteq Y,\;
\;YM_1M_1^*\subseteq Y.$$ By Lemma \ref{soson}, we may assume that $Y$ is nondegenerate. Suppose that $p$ is the identity of 
$[M_2M_2^*]^{-w^*}$ and $q$ is the identity of $[M_1M_1^*]^{-w^*}.$ Then, we have that $pY\subseteq Y, Yq\subseteq Y.$ We denote
 $A, B, C, D, \Omega (Z)$ as in Lemma \ref{xx}, and $$\Omega (pYq)=\left(\begin{array}{clr} C & pYq \\ 0 & D \end{array}\right) .$$
 If $c\in C$, then define map $$\psi (Y)\rightarrow \psi (Y): \psi(y)\rightarrow \psi (cpy).$$ Clearly, this map belongs to $A_l(Y)$, 
and thus there exists $\gamma (c)\in A_l(Y)$ satisfying $$\gamma(c) \psi (y)=\psi (cpy)\;\;\forall \;y\;\in \;Y. $$ Note that we can 
define a $*$-homomorphism $\gamma : C\rightarrow A_l(Y).$ If $\gamma (c)=0$, then $cy=0$ for all $y\in Y$, and thus because $Y$ is 
nondegenerate, it follows that $c=0.$ Therefore, $\gamma $ is one-to-one. Similarly, there exists a one-to-one $*$-homomorphism 
$$\delta : D\rightarrow A_r(Y)$$ such that $$\psi (y)\delta (d)=\psi (yqd)\;\;\forall \;y\;\in \;Y. $$ We can conclude that there 
exist projections $\hat p\in A_l(Y), \hat q\in A_r(Y)$ such that $$\psi (py)=\hat p\psi(y), \;\;\; \psi (yq)=\psi (y)\hat q,
\;\;\;\forall \;y\;\in \;Y.$$ The map $\pi : \Omega (pYq)\rightarrow \pi (\Omega(pYq)) ,$ given by $$\pi \left(\left(\begin{array}{clr}
 c & pyq \\ 0 & d \end{array}\right) \right)= \left(\begin{array}{clr} \gamma (c)& \psi (pyq) \\ 0 & \delta (d)\end {array}\right), $$ 
is a $w^*$-continuous completely isometric homomorphism. Because $$\Omega (Z)\sim _{TRO}\Omega (pYq),$$ as in Lemma \ref{xx} we can 
find a $w^*$-continuous completely isometric map $\zeta : Z\rightarrow \zeta (Z)$ and TROs $N_1, N_2$ such that $$\zeta (Z)=[N_2\psi 
(pYq)N_1^*] ^{-w^*} , \;\;\;\psi (pYq)=[N_2^*\zeta (Z)N_1] ^{-w^*}\subseteq \psi (Y) ,$$ $$\gamma (C)=[N_2^*N_2]^{-w^*},\;\;\;\delta (D)
=[N_1^*N_1]^{-w^*}.$$ Because $$N_2^*N_2\psi (Y)\subseteq \gamma (C)\psi (Y)=\psi (CY)\subseteq \psi (Y),$$ and similarly
 $\psi(Y)N_1^*N_1\subseteq \psi (Y),$ we have that $\zeta (Z) \subset _{TRO} \psi (Y).$ \end{proof}

 \begin{theorem}\label{21} Let 
$X, Y, Z$ be dual operator spaces such that $Z\subset _\Delta Y, \;\;Y\subset _\Delta X.$ Then: 

(i) There exist a $w^*$ continuous 
complete isometry $\chi: X\rightarrow \chi (X)$ and projections $p, q$ such that $p\chi (X)\subseteq \chi (X),\;\;\; \chi (X)q
\subseteq \chi (X)$ and $ Z\subset _\Delta p\chi (X)q .$ 

(ii) $Z\subset _{cb\Delta }X.$ \end{theorem}
 \begin{proof} Suppose 
that $H_1, H_2$ are Hilbert spaces such that $$A_l(Y)\subseteq B(H_2), \;\;\;A_r(Y)\subseteq B(H_1)$$ and $\psi: Y\rightarrow 
B(H_1, H_2)$ is a $w^*$-continuous complete isometry such that $$A_l(Y)\psi (Y)A_r(Y)\subseteq \psi (Y).$$ By Lemma \ref{270000}, 
there exists a $w^*$-continuous complete isometry $\zeta : Z\rightarrow \zeta (Z)$ such that $$\zeta (Z)\subset _{TRO}\psi (Y).$$ 
Because $Y\subset _\Delta X,$ there exist a $w^*$-continuous complete isometry $\chi : X\rightarrow \chi (X)$ and projections $p,q$ 
such that $p\chi (X)\subseteq \chi (X),\;\;\chi (X)q\subseteq \chi (X)$ and $$Y\sim _\Delta p\chi (X)q.$$ By Lemma \ref{xx}, there 
exists a $w^*$-continuous complete isometry $$\phi: p\chi (X)q\rightarrow \phi (p\chi (X) q)$$ such that $$\psi (Y) \sim _{TRO} \phi 
(p\chi (X)q) .$$ Thus, by Remark \ref{350000} it holds that $$\zeta (Z)\subset _{TRO}\phi (p\chi (X)q) \Rightarrow 
Z\subset _\Delta p\chi (X)q .$$ There exist projections $r,s$ and TROs $M_1, M_2$ such that $$ r\phi (p\chi (X)q)\subseteq 
\phi (p\chi (X)q), \;\;\phi (p\chi (X)q)s\subseteq \phi (p\chi (X)q) $$ and $$\zeta (Z)=[M_2 r\phi (p\chi(X)q) s M_1^*]^{-w^*},
\;\;r\phi (p\chi(X)q) s =[M_2^*\zeta (Z)M_1]^{-w^*}.$$ For every $x\in X$, define $$\hat \phi (x)=r\phi (p\chi(x)q)s\oplus r^\bot
 \phi (p\chi(x)q)s\oplus \phi (p\chi (x)q)s^\bot \oplus p^\bot \chi (x)q\oplus \chi (x)q^\bot . $$ As in the proof of Lemma \ref{19}, 
we can see that $\hat \phi $ is a $w^*$-c.b. isomorphism from $X$ onto $\hat \phi (X)$. Define the TROs $$N_i=(M_i\; 0\;0\;0\;0),
\;\; i=1,2.$$ We can see that $$[N_2\hat \phi (X)N_1^*]^{-w^*}=\zeta (Z),\;\;N_2^*\zeta (Z)N_1\subseteq \hat \phi (X)$$ and 
$$ N_2^*N_2\hat \phi (X)\subseteq \hat \phi (X) ,\;\;\hat \phi (X)N_1^*N_1\subseteq \hat \phi (X). $$ Thus, $$\zeta (Z)\subset
 _{TRO}\hat \phi (X)\Rightarrow Z\subset _{cb\Delta }X.$$ \end{proof}

 \begin{example}\label{c} Let $Y$ be a dual operator space,
 and let $H$ be a Hilbert space such that $Y\subset _\Delta B(H).$ Then, $Y\sim _\Delta B(H).$ \end{example} 
\begin{proof} There
 exist $w^*$-continuous completely isometric maps $\psi : Y\rightarrow \psi (Y)$ and projections $p, q$ such that $p\phi (B(H))
\subseteq \phi (B(H)), \phi (B(H))q\subseteq \phi (B(H))$ and $$\psi (Y)\sim _{TRO}p\phi (B(H))q.$$ We define the map $\alpha :
 B(H)\rightarrow B(H)$ given by $\alpha (x)=\phi ^{-1}(p\phi (x)).$ This is a multiplier of $B(H)$, and also a projection. Because 
$A_l(B(H))=B(H)$, there exists a projection $\hat p\in B(H)$ such that $$\phi^{-1}(p \phi(x))=\hat px\Rightarrow \phi 
(\hat px)=p\phi (x)\;\;\;\forall \;x\;\in \;B(H). $$ Similarly, there exists a projection $\hat q\in B(H)$ such that $$ 
\phi (x\hat q)=\phi (x)q\;\;\;\forall \;x\;\in \;B(H). $$ We have that $$\phi ^{-1}(p\phi (B(H))q)= \hat p B(H)\hat q .$$ Because 
$$A_l( \hat p B(H)\hat q )=B(\hat p(H)),\;\;\;A_r(\hat p B(H)\hat q )=B(\hat q(H)), $$ it follows from Lemma \ref{xx} that there 
exists a $w^*$-continuous completely isometric map $\zeta : \psi (Y)\rightarrow \zeta( \psi(Y)) $ such that $$\zeta (\psi (Y))\sim
 _{TRO}\hat p B(H)\hat q .$$ Thus, there exist TROs $M_1, M_2$ such that $$\zeta (\psi (Y))=[M_2\hat p B(H)\hat q M_1^* ]^{-w^*}.$$ 
Define $$N_2^*=[M_2\hat pB(H)]^{w^*}, \;\;\;N_1=[B(H)\hat q M_1^*]^{-w^*}.$$ Then, we have that
 $$[N_2N_2^*N_2]^{w^*}=[B(H)\hat p M_2^*M_2\hat p B(H)\hat p M_2^*]^{w^*}=[B(H)pM_2^*]^{-w^*}=N_2.$$ Thus,
 $N_2$ is a TRO. Similarly, $N_1$ is a TRO. Now, $$ \zeta( \psi (Y)) =[N_2^*B(H)N_1]^{-w^*}$$ and 
$$[N_2B(H)N_1^*]^{-w^*}=[B(H)\hat p M_2^*B(H)M_1\hat q B(H)]^{-w^*}=B(H).$$ Thus, $$\zeta( \psi (Y)) 
\sim _{TRO}B(H)\Rightarrow Y\sim _\Delta B(H). $$ \end{proof}

 \begin{example}\label{stav} \em{Let $H$ be a
 Hilbert space, $\Phi : B(H)\rightarrow B(H)$ be a $w^*$-continuous completely bounded idempotent map, and
 $$Y=Ran \Phi , \;\;Z=Ran \Phi ^\bot .$$ Let $\phi : B(H)\rightarrow Y\oplus Z$ be the map given by $\phi (x)=\Phi (x)
\oplus \Phi^\bot (x) .$ We can easily prove that $\phi $ is a $w^*$-c.b. isomorphism onto $Y\oplus Z.$ By Example \ref{examp} (ii), 
we have that $Y\subset _{TRO}\phi (B(H))$, and thus $Y\subset_{cb\Delta }B(H).$} \end{example} 

\begin{example}\label{gias} 
\em{Let $H$ be a Hilbert space, and let $e,f\in B(H)$ be nontrivial projections. Define $$\Phi : B(H)\rightarrow B(H), \;\;
\Phi (x)=exf+exf^\bot +e^\bot x f^\bot, $$ and denote $Y=Ran \Phi .$ By Example \ref{stav}, we have that $Y\subset_{cb\Delta }
B(H).$ If $Y$ weakly $\Delta $-embeds into $B(H)$, then by Example \ref{c} we should have that $Y\sim _\Delta B(H).$ However, 
this contradicts the fact that $B(H)$ is a self-adjoint algebra and $Y$ is a non-self-adjoint algebra. Thus, the relation 
$Y\subset _{cb\Delta }X$ does not always imply that $Y\subset _{\Delta }X$ holds.} \end{example} 

\bigskip {\em Acknowledgement:} 
I would like to express appreciation to Dr Evgenios Kakariadis for his helpful comments and suggestions during the preparation 
of this work.


\begin{thebibliography}{99} \bibitem{bk} D. P. Blecher and U. Kashyap, {\em Morita equivalence of dual operator algebras}, 
J. Pure Appl. Algebra, 212 (2008), 2401-2412 

\bibitem{blekra} D. P. Blecher and J. E. Kraus, {\em On a generalization of W*-modules}, Banach Center Publ. 91, 
(2010), 77-86 

\bibitem{bm} D. P. Blecher and C. Le Merdy, {\em Operator Algebras and Their Modules---An Operator Space Approach}, 
Oxford University Press, 2004 

\bibitem{bmp} D. P. Blecher, P. S. Muhly, and V. I. Paulsen, Categories of operator modules--
-Morita equivalence and projective modules,
 {\em Memoirs of the A.M.S.} 143 (2000) no. 681 

\bibitem{zar} D. P. Blecher and V. Zarikian, 
The calculus of one sided M-ideals and multipliers in operator spaces, {\em Memoirs of the A.M.S.} 179 (842), 
2003 

\bibitem{dav} K. R. Davidson, \newblock {\em Nest Algebras}, \newblock Longman Scientific \& Technical, Harlow, 1988 

\bibitem{er} E. Effros and Z.-J. Ruan, {\em Operator Spaces}, Oxford University Press, 2000 

\bibitem{ele1} G.K. Eleftherakis, TRO equivalent algebras, {\em Houston J. of Mathematics}, 38:1 (2012), 153-175 

\bibitem{ele2} G.K. Eleftherakis, A Morita type equivalence for dual operator algebras, {\em J. Pure Appl. Algebra}, 212:5 (2008),
 1060-1071 

\bibitem{ele3} G.K. Eleftherakis, Morita type equivalences and reflexive algebras, 
{\em J. Operator Theory}, 64 (2010) no 1, 3-17

 \bibitem{elenest} G.K. Eleftherakis, Morita equivalence of nest algebras, 
{\em Math. Scand.}, 113 (2013), no 1, 83-107 

\bibitem{ele4} G.K. Eleftherakis, V. I. Paulsen, Stably isomorphic dual operator algebras, 
{Math. Ann.}, 341:1 (2008), 99-112 

\bibitem{ele5} G. K. Eleftherakis, V. I. Paulsen, and I. G. Todorov, 
Stable isomorphism of dual operator spaces, {\em J. Funct. Anal.} 258 (2010), 260--278 

\bibitem{kashyap} U. Kashyap, A Morita theorem for dual operator algebras, J. Funct. Analysis, 256 (2009) 3545-3567 

\bibitem{paul} V. I. Paulsen, {\em Completely Bounded Maps and Operator Algebras},
 Cambridge University Press, 2002 

\bibitem{p} G. Pisier, {\em Introduction to Operator Space Theory}, Cambridge University Press, 2000 

\end{thebibliography}
\end{document}